\documentclass[a4paper]{amsart}
\usepackage{amscd}
\usepackage{amsmath}
\usepackage{mathrsfs}
\usepackage{amssymb}
\usepackage{amsthm}
\usepackage{bbm}
\usepackage{stmaryrd}

\usepackage[T1]{fontenc}

\makeatletter
\@namedef{subjclassname@2020}{%
  \textup{2020} Mathematics Subject Classification}
\makeatother

\newif\ifpdf
\ifx\pdfoutput\undefined
   \pdffalse        
\else
   \pdfoutput=1     
   \pdftrue
\fi

\ifpdf
   \usepackage[pdftex]{graphicx}
\else
   \usepackage{graphicx}
\fi

\numberwithin{equation}{section} \swapnumbers

\newtheorem{satz}{Satz}[section]

\newtheorem{theorem}[satz]{Theorem}
\newtheorem{proposition}[satz]{Proposition}
\newtheorem{corollary}[satz]{Corollary}
\newtheorem{lemma}[satz]{Lemma}

\newtheorem{remark}[satz]{Remark}

\newtheorem{example}[satz]{Example}

\newcommand{\bbr}{\mathbb{R}}
\newcommand{\bbe}{\mathbb{E}}
\newcommand{\bbn}{\mathbb{N}}
\newcommand{\bbp}{\mathbb{P}}
\newcommand{\bbq}{\mathbb{Q}}

\newcommand{\calb}{\mathcal{B}}

\newcommand{\calf}{\mathcal{F}}

\newcommand{\call}{\mathcal{L}}
\newcommand{\calm}{\mathcal{M}}

\newcommand{\ran}{{\rm ran}}

\newcommand{\Id}{{\rm Id}}
\newcommand{\N}{{\rm N}}
\newcommand{\Var}{{\rm Var}}
\newcommand{\tr}{{\rm tr}}
\newcommand{\lin}{{\rm lin}}

\newcommand{\la}{\langle}
\newcommand{\ra}{\rangle}

\newcommand{\bbI}{\mathbbm{1}}

\begin{document}

\title[Linear estimators for Gaussian random variables in Hilbert spaces]{Linear estimators for Gaussian random variables in Hilbert spaces}
\author{Stefan Tappe}
\address{Albert Ludwig University of Freiburg, Department of Mathematical Stochastics, Ernst-Zermelo-Stra\ss{}e 1, D-79104 Freiburg, Germany}
\email{stefan.tappe@math.uni-freiburg.de}
\address{University of Wuppertal, Department of Mathematics and Natural Sciences, Gau\ss{}-stra\ss{}e 20, D-42097 Wuppertal, Germany}
\email{tappe@uni-wuppertal.de}
\date{7 June, 2024}
\thanks{The author is grateful to Ludwig Baringhaus, Lyudmila Grigoryeva and Juan-Pablo Ortega for fruitful discussions. Moreover, the author gratefully acknowledges financial support from the Deutsche Forschungsgemeinschaft (DFG, German Research Foundation) --- project number 444121509.}
\begin{abstract}
We study a statistical model for infinite dimensional Gaussian random variables with unknown parameters. For this model we derive linear estimators for the mean and the variance of the Gaussian distribution. Furthermore, we construct confidence intervals and perform hypothesis testing. A linear regression problem in infinite dimensions and some perspectives to statistical and machine learning are presented as well.
\end{abstract}
\keywords{Gaussian random variable in a Hilbert space, statistical model, linear estimator, confidence interval, hypothesis testing, linear regression, statistical and machine learning}
\subjclass[2020]{62F03, 62F10, 62F25, 62J05, 28C20}

\maketitle

\section{Introduction}\label{sec-intro}

In this paper we study linear estimators for Gaussian random variables with values in a separable Hilbert space. More precisely, given a separable Hilbert space $H$, a closed subspace $U \subset H$, and a linear operator $Q \in L(H)$, the statistical model is of the form
\begin{align}\label{model}
\calm = \big( H, \calb(H), (\bbp_{\vartheta} : \vartheta \in \Theta) \big),
\end{align}
where the parameter domain is given by $\Theta := U \times (0,\infty)$, and where the probability measures are the Gaussian measures
\begin{align}\label{fam-measures}
\bbp_{\vartheta} := \N(\zeta,\sigma^2 Q), \quad \vartheta = (\zeta,\sigma) \in \Theta.
\end{align}
Hence, in this model we consider random variables
\begin{align}\label{observation}
Y = \zeta + \sigma \epsilon
\end{align}
with an infinite dimensional noise $\epsilon \sim {\rm N}(0,Q)$ and unknown parameters $(\zeta,\sigma) \in \Theta$. We are in particular interested in estimating the mean $\zeta \in U$ from an observation $Y$. In (\ref{observation}), the noise $\epsilon$ could be, for example, a centered Gaussian process.

There is a well-established literature about linear models in finite dimensions; see, for example \cite{Myers-Milton, Wichura} or Chapter 7 in the German textbook \cite{Czado-Schmidt}. However, so far there are only few references (such as \cite{Mandelbaum-0, Mandelbaum, Luschgy, Ghiglietti, Mollenhauer}) dealing with linear estimators or linear regression problems for infinite dimensional random variables. The goal of this paper is to generalize the findings about linear models in finite dimensions to the present infinite dimensional model (\ref{model}). This includes the construction of estimators for the mean and the variance, confidence intervals, hypothesis testing, and a linear regression problem in infinite dimensions. Some perspectives to statistical and machine learning are presented as well.

Our results are not straightforward, because for an infinite dimensional Hilbert space $H$ the covariance operator $Q$ must be a trace class operator, which excludes the identity operator. As a consequence, the Gaussian random variable $Y$ has a series representation of the form
\begin{align}\label{series-Gauss-intro}
Y = \zeta + \sigma \sum_{k \in K} \sqrt{\lambda_k} \beta_k e_k
\end{align}
with an orthonormal system $(e_k)_{k \in K}$, positive elements $(\lambda_k)_{k \in K} \subset (0,\infty)$ such that $\sum_{k \in K} \lambda_k < \infty$, and a sequence $(\beta_k)_{k \in K}$ of independent and identically distributed random variables such that $\beta_k \sim \N(0,1)$ for each $k \in K$. More precisely, the sequence $(e_k)_{k \in K}$ consists of eigenvectors of $Q$ with corresponding eigenvalues $(\lambda_k)_{k \in K}$. The representation (\ref{series-Gauss-intro}) shows that the distribution of $\| Y - \zeta \|^2$ is no longer a $\chi^2$-distribution. Its density is not available in closed form, and there is only an asymptotic formula; see \cite{Zolotarev}. We will overcome this difficulty by performing estimates depending on the largest eigenvalue of the operator $Q(\Pi_H - \Pi_U)$ and the dimension of its eigenspace.

The remainder of this paper is organized as follows. In Section \ref{sec-model} we introduce the statistical model. In Section \ref{sec-est-mean} we construct an estimator for the mean and discuss its properties. In Section \ref{sec-ML} we discuss some practical aspects and draw perspectives to statistical and machine learning. In Section \ref{sec-est-mean-2} we construct an estimator for functionals of the mean and present a generalization of the Gauss-Markov theorem to infinite dimensions. In Section \ref{sec-conf-known} we construct confidence intervals for the mean in the particular case that the variance is known. In Section \ref{sec-est-var} we construct an estimator for the variance and discuss its properties. In Section \ref{sec-conf-unknown} we construct confidence intervals for the mean in the general situation where the variance is unknown. In Section \ref{sec-test} we construct a statistical test for the hypothesis that the mean belongs to a smaller subspace. In Section \ref{sec-learning} we treat a linear regression problem in infinite dimensions, which may also have applications to statistical and machine learning. For convenience of the reader, in Appendix \ref{app-operators} we provide the required results about linear operators in Hilbert spaces, in Appendix \ref{app-Gaussian} we provide the required results about Gaussian random variables in Hilbert spaces, and in Appendix \ref{app-distributions} we provide the required results about the Student's $t$-distribution and the Fisher distribution.

\section{The statistical model}\label{sec-model}

In this section we introduce the statistical model. Let $(H, \langle \cdot,\cdot \rangle)$ be a separable Hilbert space. We fix a closed subspace $U \subset H$. Typically, this subspace is chosen such that $U \neq \{ 0 \}$ and $U \neq H$. Let $Q \in L_1^+(H)$ be a nonnegative self-adjoint nuclear operator such that $U$ and $U^{\perp}$ are $Q$-invariant. We refer to Appendix \ref{app-operators} for further details about linear operators in Hilbert spaces.

Consider the statistical model (\ref{model}), where the parameter domain
is given by $\Theta := U \times (0,\infty)$, and where the probability measures are the Gaussian measures given by (\ref{fam-measures}). In other words, we have an $H$-valued Gaussian random variable $Y$ with unknown parameters. More precisely, for each $\vartheta = (\zeta,\sigma) \in \Theta$ we have $\bbp_{\vartheta} \circ Y^{-1} = \N(\zeta,\sigma^2 Q)$. We refer to Appendix \ref{app-Gaussian} for further details about Gaussian random variables in Hilbert spaces.

\begin{remark}\label{rem-pivot}
Note that for each $\vartheta = (\zeta, \sigma) \in \Theta$ we have
\begin{align*}
\bbp_{\vartheta} \circ \bigg( \frac{Y - \zeta}{\sigma} \bigg)^{-1} = \N(0,Q).
\end{align*}
Hence for every linear operator $T \in L(H)$ with $U \subset \ker(T)$ the random variable
\begin{align*}
T \bigg( \frac{Y - \zeta}{\sigma} \bigg) = T \bigg( \frac{Y}{\sigma} \bigg)
\end{align*}
is a pivot statistics.
\end{remark}

As a consequence of the Karhunen-Lo\`{e}ve theorem, the statistical model (\ref{model}) covers Gaussian processes; see, for example \cite[Thm. 4.6]{Giambartolomei}. More precisely, let $X$ be a centered Gaussian process.
Consider $H := L^2([0,1])$ and the integral operator $Q \in L(H)$ given by
\begin{align}\label{integral-operator}
Qh(s) = \int_0^1 k(s,t) h(t) dt, \quad h \in H,
\end{align}
where the kernel $k \in L^2 \big( [0,1]^2 \big)$ denotes the covariance function of the Gaussian process $X$. Then the random variable (\ref{observation}) can be written as
\begin{align}\label{Gaussian-observation}
Y = \zeta + \sigma X
\end{align}
with unknown parameters $(\zeta,\sigma) \in \Theta$. Hence, an observation $Y$ is a superposition of an unknown signal $\zeta \in U$ and the perturbation $\sigma X$, and we are interested in estimating $\zeta$ from this observation.

\begin{example}[Wiener process]\label{ex-Wiener}
Suppose that $Q$ is the integral operator (\ref{integral-operator}) with kernel $k(s,t) = \min \{ s,t \}$ for all $s,t \in [0,1]$. Then $Q \in L_1^{+}(H)$ is the nonnegative self-adjoint nuclear operator with eigenvalues and eigenvectors given by
\begin{align}\label{Wiener-eigenvalues}
\lambda_k &= \frac{1}{(k - \frac{1}{2})^2 \pi^2}, \quad k \in \bbn,
\\ \label{Wiener-eigenvectors} e_k &= \sqrt{2} \, \sin \bigg( \Big( k - \frac{1}{2} \Big) \pi \bullet \bigg), \quad k \in \bbn.
\end{align}
Note that $( e_k )_{k \in \bbn}$ is an orthonormal system, which is not complete. Indeed, for every $h \in \overline{\lin} \{ e_k : k \in \bbn \}$ we have $h(0) = 0$, which implies that $Q$ is not strictly positive. By \cite[Thm. 5.14]{Giambartolomei} the series
\begin{align*}
W = \sum_{k=1}^{\infty} \sqrt{\lambda_k} \beta_k e_k
\end{align*}
is a standard Wiener process on the interval $[0,1]$. Therefore, the random variable (\ref{Gaussian-observation}) can be written as
\begin{align}\label{Wiener-observation}
Y = \zeta + \sigma W
\end{align}
with a Wiener process $W$ and unknown parameters $(\zeta,\sigma) \in \Theta$. Using (\ref{Wiener-eigenvalues}), the trace of the covariance operator is given by
\begin{align}\label{Wiener-trace}
\tr(Q) = \sum_{k=1}^{\infty} \lambda_k = \frac{4}{\pi^2} \sum_{k=1}^{\infty} \frac{1}{(2k-1)^2} = \frac{4}{\pi^2} \cdot \frac{\pi^2}{8} = \frac{1}{2}.
\end{align}
\end{example}

\begin{example}[Brownian bridge]
Suppose that $Q$ is the integral operator (\ref{integral-operator}) with kernel $k(s,t) = \min \{ s,t \} - st$ for all $s,t \in [0,1]$. Then $Q \in L_1^{+}(H)$ is the nonnegative self-adjoint nuclear operator with eigenvalues and eigenvectors given by
\begin{align}\label{Bridge-eigenvalues}
\lambda_k &= \frac{1}{k^2 \pi^2}, \quad k \in \bbn,
\\ \label{Bridge-eigenvectors} e_k &= \sqrt{2} \, \sin ( k \pi \bullet ), \quad k \in \bbn.
\end{align}
Note that $( e_k )_{k \in \bbn}$ is an orthonormal system, which is not complete. Indeed, for every $h \in \overline{\lin} \{ e_k : k \in \bbn \}$ we have $h(0) = h(1) = 0$, which implies that $Q$ is not strictly positive. By \cite[Thm. 5.17]{Giambartolomei} the series
\begin{align*}
B = \sum_{k=1}^{\infty} \sqrt{\lambda_k} \beta_k e_k
\end{align*}
is a Brownian bridge on the interval $[0,1]$. Therefore, the random variable (\ref{Gaussian-observation}) can be written as
\begin{align*}
Y = \zeta + \sigma B
\end{align*}
with a Brownian bridge $B$ and unknown parameters $(\zeta,\sigma) \in \Theta$. Using (\ref{Bridge-eigenvalues}), the trace of the covariance operator is given by
\begin{align*}
\tr(Q) = \sum_{k=1}^{\infty} \lambda_k = \frac{1}{\pi^2} \sum_{k=1}^{\infty} \frac{1}{k^2} = \frac{1}{\pi^2} \cdot \frac{\pi^2}{6} =  \frac{1}{6}.
\end{align*}
\end{example}

Now, let us return to the statistical model (\ref{model}). Let $T : H \to H$ be a measurable mapping such that $T(Y) \in L^2(\bbp_{\vartheta})$ for each $\vartheta \in \Theta$. For the estimator $T(Y)$ we introduce the following notions:
\begin{itemize}
\item We call
\begin{align*}
B_{\vartheta}[T(Y)] := \| \bbe_{\vartheta}[T(Y)] - \zeta \|, \quad \vartheta \in \Theta
\end{align*}
the \emph{bias} of $T(Y)$.

\item We call
\begin{align*}
\Var_{\vartheta}[T(Y)] := \bbe_{\vartheta} \big[ \| T(Y) - \bbe_{\vartheta}[T(Y)] \|^2 \big], \quad \vartheta \in \Theta
\end{align*}
the \emph{variance} of $T(Y)$.

\item We call
\begin{align*}
R_{\vartheta}[T(Y)] := \bbe_{\vartheta} \big[ \| T(Y) - \zeta \|^2 \big], \quad \vartheta \in \Theta
\end{align*}
the \emph{risk} (or \emph{mean squared error}) of $T(Y)$.
\end{itemize}

\begin{lemma}\label{lemma-risk-decomposition}
For each $\vartheta \in \Theta$ we have the decomposition
\begin{align*}
R_{\vartheta}[T(Y)] = \Var_{\vartheta}[T(Y)] + B_{\vartheta}[T(Y)]^2.
\end{align*}
In particular, if $T(Y)$ is an unbiased estimator, then for each $\vartheta \in \Theta$ we have
\begin{align*}
R_{\vartheta}[T(Y)] = \Var_{\vartheta}[T(Y)].
\end{align*}
\end{lemma}

\begin{proof}
A simple calculation shows that
\begin{align*}
R_{\vartheta}[T(Y)] &= \bbe_{\vartheta} \big[ \| T(Y) - \zeta \|^2 \big]
\\ &= \bbe_{\vartheta} \big[ \| T(Y) - \bbe_{\vartheta}[T(Y)] + \bbe_{\vartheta}[T(Y)] - \zeta \|^2 \big]
\\ &= \bbe_{\vartheta} \big[ \| T(Y) - \bbe_{\vartheta}[T(Y)] \|^2 + 2 \la T(Y) - \bbe_{\vartheta}[T(Y)], \bbe_{\vartheta}[T(Y)] - \zeta \ra
\\ &\qquad \quad + \| \bbe_{\vartheta}[T(Y)] - \zeta \|^2 \big]
\\ &= \bbe_{\vartheta} \big[ \| T(Y) - \bbe_{\vartheta}[T(Y)] \|^2 \big] + 2 \la \bbe_{\vartheta} [ T(Y) - \bbe_{\vartheta}[T(Y)] ], \bbe_{\vartheta}[T(Y)] - \zeta \ra
\\ &\quad + \| \bbe_{\vartheta}[T(Y)] - \zeta \|^2
\\ &= \Var_{\vartheta}[T(Y)] + B_{\vartheta}[T(Y)]^2,
\end{align*}
completing the proof.
\end{proof}

\section{Estimation of the mean}\label{sec-est-mean}

In this section we estimate the mean $\zeta$ of the statistical model (\ref{model}). We define $\widehat{\zeta} \in L(H)$ as the orthogonal projection $\widehat{\zeta} := \Pi_U$ and consider the estimator $\widehat{\zeta}(Y)$. As an immediate consequence of Corollary \ref{cor-proj-Gauss} we obtain:

\begin{proposition}\label{prop-zeta}
For each $\vartheta = (\zeta,\sigma) \in \Theta$ we have $\bbp_{\vartheta} \circ \widehat{\zeta}(Y)^{-1} = {\rm N}(\zeta,\sigma^2 Q \Pi_U)$.
\end{proposition}

In particular, $\widehat{\zeta}(Y)$ is an unbiased estimator for $\zeta$. Moreover, as an immediate consequence of Proposition \ref{prop-zeta}, Lemma \ref{lemma-risk-decomposition} and Proposition \ref{prop-Gauss-rv} we obtain:

\begin{lemma}\label{lemma-risk-mean}
We have $R_{\vartheta}[\widehat{\zeta}(Y)] = \sigma^2 \tr(Q \Pi_U)$ for each $\vartheta = (\zeta,\sigma^2) \in \Theta$.
\end{lemma}

The following result shows that the estimator $\widehat{\zeta}(Y)$ is risk minimizing for $\zeta$.

\begin{theorem}\label{thm-risk-min}
 For each self-adjoint operator $T \in L(H)$ such that $T(Y)$ is an unbiased linear estimator for $\zeta$ we have
\begin{align*}
R_{\vartheta} [ \widehat{\zeta}(Y) ] \leq R_{\vartheta} [ T(Y) ] \quad \text{for all $\vartheta = (\zeta,\sigma^2) \in \Theta$.}
\end{align*}
\end{theorem}

\begin{proof}
Let $\vartheta = (\zeta,\sigma^2) \in \Theta$ be arbitrary. Furthermore, let $T \in L(H)$ be a self-adjoint operator such that $T(Y)$ is an unbiased linear estimator for $\zeta$. Then we have
\begin{align*}
T(\zeta) = T(\bbe_{\vartheta}[Y]) = \bbe_{\vartheta}[T(Y)] = \zeta,
\end{align*}
showing that $T|_U = \Id|_U$. Therefore, and since $T$ is self-adjoint, by Lemma \ref{lemma-self-adjoint} the subspace $U^{\perp}$ is $T$-invariant. Furthermore, by Proposition \ref{prop-Gauss-rv} we have
\begin{align*}
\bbp_{\vartheta} \circ T(Y)^{-1} = {\rm N}(T(\zeta),\sigma^2 T Q T^*) = {\rm N}(\zeta,\sigma^2 T Q T).
\end{align*}
Since $U^{\perp}$ is $Q$-invariant and $T$-invariant, we have
\begin{align*}
T Q T \Pi_{U^{\perp}} = \Pi_{U^{\perp}} T Q T \Pi_{U^{\perp}} \in L_1^+(H).
\end{align*}
Moreover, since $U$ is $Q$-invariant and $T|_U = \Id|_U$, we have
\begin{align*}
Q \Pi_U = T Q T \Pi_U.
\end{align*}
Therefore, by Lemma \ref{lemma-risk-mean} we obtain
\begin{align*}
R_{\vartheta} [ \widehat{\zeta}(Y) ] &= \sigma^2 \tr(Q \Pi_U) \leq \sigma^2 \tr(Q \Pi_U) + \sigma^2 \tr(T Q T \Pi_{U^{\perp}})
\\ &= \sigma^2 \tr(T Q T \Pi_U) + \sigma^2 \tr(T Q T \Pi_{U^{\perp}}) = \sigma^2 \tr(T Q T) = R_{\vartheta} [ T(Y) ],
\end{align*}
where we have used Lemma \ref{lemma-risk-decomposition} and Proposition \ref{prop-Gauss-rv} in the last step.
\end{proof}

\section{Perspectives to statistical and machine learning}\label{sec-ML}

In this section we discuss some practical aspects of the estimation of the mean and draw perspectives to statistical and machine learning. Note that the statistical model (\ref{model}) gives rise to a single observation $Y$ with values in the Hilbert space $H$, which is typically large. As we can see from the series representation (\ref{series-Gauss-intro}), this corresponds to a sequence of random variables, where long-term observations become decreasingly important.

In reality, the random variable $Y$ may only be partially observable; say we may only be able to observe the random variable $\Pi_V Y$ for some subspace $V \subset U$ rather than $Y$. In this situation the estimator $\Pi_V Y$ is typically biased. The corresponding risk admits the following decomposition.

\begin{proposition}\label{prop-learning-1}
Let $V \subset U$ be a closed subspace such that $V$ and $V^{\perp}$ are $Q$-invariant. Then for all $\vartheta = (\zeta,\sigma^2) \in \Theta$ we have the decomposition
\begin{align*}
R_{\vartheta} [ \Pi_V Y ] = \sigma^2 \tr (Q \Pi_V) + \| \Pi_{V^{\perp}} \zeta \|^2.
\end{align*}
\end{proposition}

\begin{proof}
Let $\vartheta = (\zeta,\sigma^2) \in \Theta$ be arbitrary. By Proposition \ref{prop-Gauss-rv} we have
\begin{align*}
\bbp_{\vartheta} \circ (\Pi_V Y)^{-1} = \N(\Pi_V \zeta,\sigma^2 Q \Pi_V).
\end{align*}
Therefore, by Lemma \ref{lemma-risk-decomposition} and Proposition \ref{prop-Gauss-rv} we obtain
\begin{align*}
R_{\vartheta} [ \Pi_V Y ] &= \Var_{\vartheta}[ \Pi_V Y ] + B_{\vartheta}[ \Pi_V Y ]^2
\\ &= \sigma^2 \tr (Q \Pi_V) + \| \Pi_V \zeta - \zeta \| = \sigma^2 \tr (Q \Pi_V) + \| \Pi_{V^{\perp}} \zeta \|^2,
\end{align*}
completing the proof.
\end{proof}

Now, suppose that the closed subspace $U$ is infinite dimensional. Furthermore, we assume there are an orthonormal system $(f_k)_{k \in \bbn}$ and positive elements $(\mu_k)_{k \in \bbn} \subset (0,\infty)$ with $\sum_{k=1}^{\infty} \mu_k < \infty$ such that
\begin{align*}
U = \overline{\lin} \{ f_k : k \in \bbn \}
\end{align*}
as well as
\begin{align*}
Q f_k = \mu_k f_k \quad \text{for all $k \in \bbn$.}
\end{align*}

Note that in view of Remark \ref{rem-spectral} these are natural conditions.

\begin{example}
Consider the statistical model (\ref{model}) from Example \ref{ex-Wiener} with
\begin{align*}
U := \overline{\lin} \{ e_{2k} : k \in \bbn \}.
\end{align*}
Then the sequences $(f_k)_{k \in \bbn}$ and $(\mu_k)_{k \in \bbn}$ are given by
\begin{align*}
f_k &= \sqrt{2} \, \sin \bigg( \Big( 2k - \frac{1}{2} \Big) \pi \bullet \bigg),
\\ \mu_k &= \frac{1}{(2k - \frac{1}{2})^2 \pi^2}.
\end{align*}
\end{example}

Now, we define the subspaces
\begin{align*}
V_n := \lin \{ f_1,\ldots,f_n \}, \quad n \in \bbn,
\end{align*}
and consider the estimators
\begin{align*}
\widehat{\zeta}_n(Y) := \Pi_{V_n} Y, \quad n \in \bbn,
\end{align*}
which give rise to an algorithm. In statistical and machine learning the question arises how fast this algorithm is learning. Recalling Theorem \ref{thm-risk-min}, this issue is addressed by the following result.

\begin{proposition}
For each $\vartheta = (\zeta,\sigma^2) \in \Theta$ we have
\begin{align*}
R_{\vartheta}[\widehat{\zeta}_n(Y)] - R_{\vartheta} [\widehat{\zeta}(Y)] = \sum_{k=n+1}^{\infty} ( \zeta_k^2 - \sigma^2 \lambda_k ) \to 0 \quad \text{for $n \to \infty$,}
\end{align*}
where $\zeta_k := \la \zeta,e_k \ra$ for each $k \in \bbn$.
\end{proposition}

\begin{proof}
Let $n \in \bbn$ be arbitrary. By Proposition \ref{prop-learning-1} we have
\begin{align*}
R_{\vartheta}[\widehat{\zeta}_n(Y)] = \sigma^2 \tr (Q \Pi_{V_n}) + \| \Pi_{V_n^{\perp}} \zeta \|^2 = \sigma^2 \sum_{k=1}^n \lambda_k + \sum_{k=n+1}^{\infty} \zeta_k^2.
\end{align*}
Furthermore, by Lemma \ref{lemma-risk-mean} we have
\begin{align*}
R_{\vartheta} [\widehat{\zeta}(Y)] = \sigma^2 \tr(Q \Pi_U) = \sigma^2 \sum_{k=1}^{\infty} \lambda_k.
\end{align*}
Combining these two identities provides the result.
\end{proof}

\section{Estimation of functionals of the mean}\label{sec-est-mean-2}

In this section we estimate linear functionals of the mean $\zeta$, given the statistical model (\ref{model}). As an immediate consequence of Corollary \ref{cor-proj-Gauss} we obtain:

\begin{proposition}\label{prop-zeta-b}
For each $\vartheta = (\zeta,\sigma) \in \Theta$ and each $b \in H$ we have
\begin{align*}
\bbp_{\vartheta} \circ \langle b,\widehat{\zeta}(Y) \rangle^{-1} = {\rm N}(\langle b,\zeta \rangle, \sigma^2 \langle Q b, \Pi_U b \rangle).
\end{align*}
\end{proposition}

In particular, for each $b \in H$ the estimator $\langle b,\widehat{\zeta}(Y) \rangle$ is an unbiased estimator for $\langle b,\zeta \rangle$. The following result may be regarded as a generalization of the Gauss-Markov theorem to infinite dimensions. Note that $\langle b,\widehat{\zeta}(Y) \rangle = \la \Pi_U b, Y \ra$ for all $b \in H$.

\begin{theorem}\label{thm-Gauss-Markov}
Let $b \in H$ be arbitrary, and let $c \in H$ such that $\la c,Y \ra$ is an unbiased estimator for $\la b,\zeta \ra$. Then the following statements are true:
\begin{enumerate}
\item We have
\begin{align}\label{GM-1}
R_{\vartheta}[\langle b,\widehat{\zeta}(Y) \rangle] \leq R_{\vartheta}[\langle c,Y \rangle] \quad \text{for all $\vartheta \in \Theta$.}
\end{align}
\item Suppose that
\begin{align}\label{GM-2}
R_{\vartheta}[\langle b,\widehat{\zeta}(Y) \rangle] = R_{\vartheta}[\langle c,Y \rangle] \quad \text{for all $\vartheta \in \Theta$.}
\end{align}
Then we have
\begin{align}\label{GM-3}
\la c,Y \ra = \langle b,\widehat{\zeta}(Y) \rangle \quad \text{$\bbp_{\vartheta}$-almost surely for each $\vartheta \in \Theta$.}
\end{align}
\end{enumerate}
\end{theorem}

\begin{proof}
Let $\vartheta = (\zeta,\sigma^2) \in \Theta$ be arbitrary. Then we have
\begin{align*}
\langle b,\zeta \rangle = \bbe_{\vartheta}[\langle c,Y \rangle] = \langle c,\zeta \rangle.
\end{align*}
Thus, we have $b-c \in U^{\perp}$, and hence $\Pi_U b = \Pi_U c$. Since $Q$ is a nonnegative operator, by Lemma \ref{lemma-risk-decomposition}, Proposition \ref{prop-zeta-b}, Lemma \ref{lemma-Q-Pi} and Proposition \ref{prop-Gauss-rv} we obtain
\begin{equation}\label{calc}
\begin{aligned}
R_{\vartheta}[\langle b,\widehat{\zeta}(Y) \rangle] = \Var_{\vartheta} [ \langle b,\widehat{\zeta}(Y) \rangle ] &= \sigma^2 \langle Q \Pi_U b,\Pi_U b \rangle = \sigma^2 \langle Q \Pi_U c,\Pi_U c \rangle
\\ &\leq \sigma^2 \langle Q c,c \rangle = \Var_{\vartheta} [ \langle c,Y \rangle ] = R_{\vartheta} [ \langle c,Y \rangle ],
\end{aligned}
\end{equation}
proving (\ref{GM-1}). In order to prove the claimed uniqueness, suppose that (\ref{GM-2}) is fulfilled. The previous calculation (\ref{calc}) shows that $\la Q \Pi_U c, \Pi_U c \ra = \la Qc, c \ra$. Since
\begin{align*}
\la Qc, c \ra = \la Q \Pi_U c, \Pi_U c \ra + \la Q \Pi_{U^{\perp}} c, \Pi_{U^{\perp}} c \ra,
\end{align*}
we obtain $\la Q \Pi_{U^{\perp}} c, \Pi_{U^{\perp}} c \ra = 0$. Hence, by Lemma \ref{lemma-operator-kernel} we have $\Pi_{U^{\perp}} c \in \ker(Q)$. Now, recalling that $\Pi_U b = \Pi_U c$, we obtain
\begin{align*}
c = \Pi_U c + \Pi_{U^{\perp}} c = \Pi_U b + \Pi_{U^{\perp}} c.
\end{align*}
Let $\vartheta = (\zeta,\sigma^2) \in \Theta$ be arbitrary. By Proposition \ref{prop-Gauss-rv} we have
\begin{align*}
\la \Pi_{U^{\perp}} c, Y \ra \sim \N(\la \Pi_{U^{\perp}} c,\zeta \ra, \sigma^2 \la Q \Pi_{U^{\perp}} c, \Pi_{U^{\perp}} c \ra ).
\end{align*}
Since $\Pi_{U^{\perp}} c \in \ker(Q)$, we obtain
\begin{align*}
\la \Pi_{U^{\perp}} c, Y \ra = 0 \quad \text{$\bbp_{\vartheta}$-almost surely,}
\end{align*}
and hence
\begin{align*}
\la c,Y \ra = \la \Pi_U b,Y \ra + \la \Pi_{U^{\perp}} c, Y \ra = \la b,\Pi_U Y \ra = \la b,\widehat{\zeta}(Y) \ra \quad \text{$\bbp_{\vartheta}$-almost surely,}
\end{align*}
proving (\ref{GM-3}).
\end{proof}

\begin{example}\label{ex-Wiener-amplitude}
Consider the statistical model (\ref{model}) with $Q$ as in Example \ref{ex-Wiener} and $U = \lin \{ h_k \}$ for some $k \in \bbn$, where
\begin{align*}
h_k = \sin \bigg( \Big( k - \frac{1}{2} \Big) \pi \bullet \bigg).
\end{align*}
Then the random variable (\ref{Wiener-observation}) is a noisy oscillation perturbed by a Wiener process with unknown variance. Figure \ref{fig-signal} shows a typical observation in case $k=4$.
\begin{figure}[!ht]
 \centering
 \includegraphics[width=0.5\textwidth]{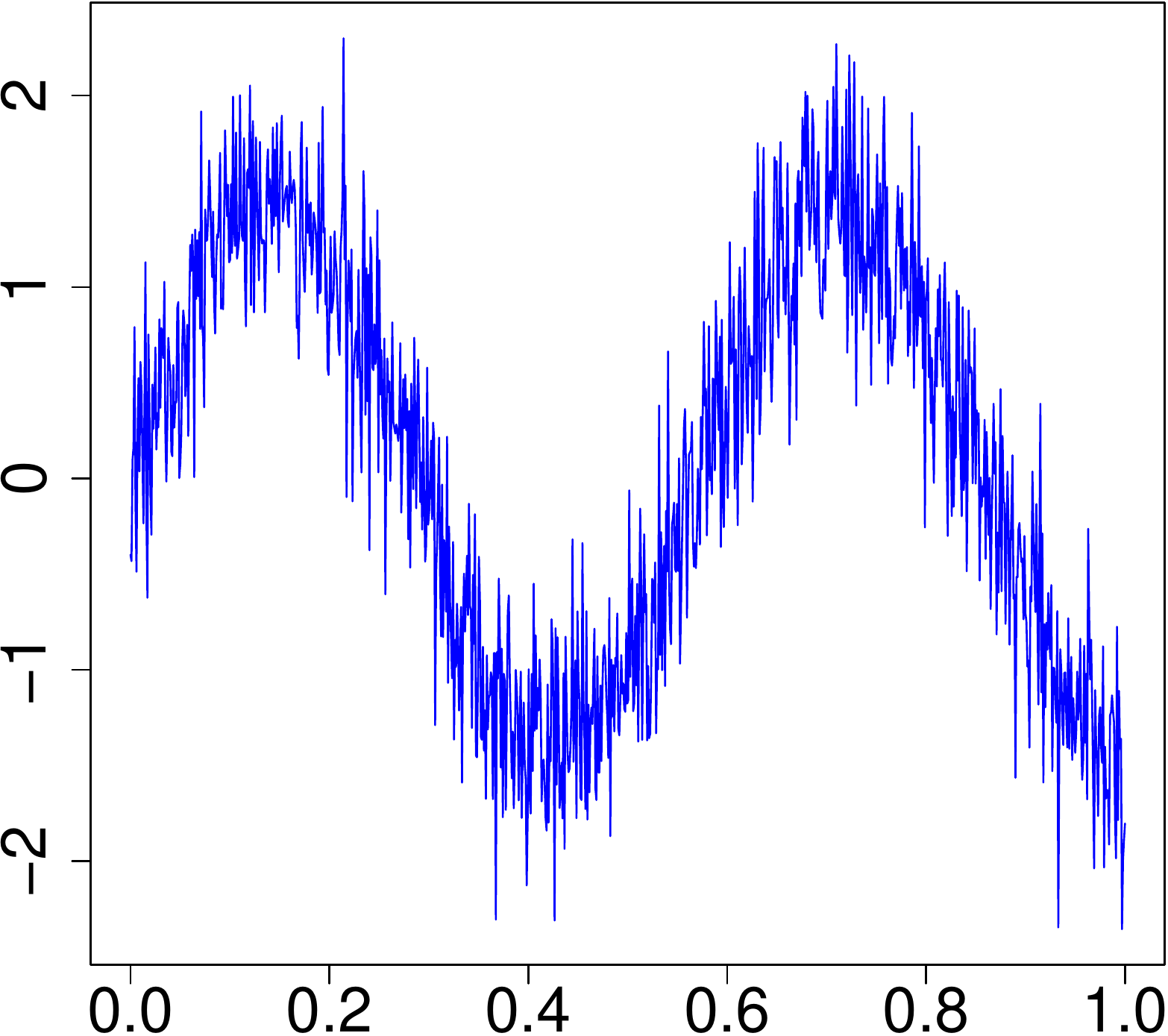}
 \caption{Observation of the random variable $Y$ in case $k=4$.}\label{fig-signal}
\end{figure}
We would like to estimate the amplitude of the oscillation. For this purpose, observe that for each $\zeta \in U$ we have $\zeta = \gamma h_k$ for some $\gamma \in \bbr$. Noting (\ref{Wiener-eigenvectors}), with $b := 2 h_k$ we obtain
\begin{align*}
\la b,\zeta \ra = \la 2 h_k, \gamma h_k \ra = \gamma \la \sqrt{2} h_k, \sqrt{2} h_k \ra = \gamma \la e_k,e_k \ra = \gamma.
\end{align*}
Therefore, according to Theorem \ref{thm-Gauss-Markov} the estimator $\langle b,\widehat{\zeta}(Y) \rangle$ is the unique risk minimizing unbiased estimator for the amplitude $\gamma$.
\end{example}

\section{Confidence intervals with known variance}\label{sec-conf-known}

In this section we construct confidence intervals for the mean $\zeta$ for a simplified statistical model with known variance. More precisely, we fix $\sigma^2 \in (0,\infty)$ and consider
\begin{align}\label{model-0}
\calm_0 = \big( H, \calb(H), \bbp_{\zeta} : \zeta \in U \big),
\end{align}
where the probability measures are given by
\begin{align}\label{family-0}
\bbp_{\zeta} := \N(\zeta,\sigma^2 Q), \quad \zeta \in U.
\end{align}
In the sequel, for each $a \in (0,1)$ we denote by $z_a$ the $a$-quantile of the standard normal distribution ${\rm N}(0,1)$.

\begin{theorem}\label{thm-conf-known}
Let $b \in H$ with $\langle Q b, \Pi_U b \rangle > 0$ and $\alpha \in (0,1)$ be arbitrary. Then the random interval
\begin{align}\label{conf-int-1}
\langle b,\widehat{\zeta}(Y) \rangle \pm z_{1-\frac{\alpha}{2}} \, \sigma \sqrt{\langle Q b, \Pi_U b \rangle}
\end{align}
is an $(1-\alpha)$-confidence interval for $\langle b,\zeta \rangle$.
\end{theorem}

\begin{proof}
Let $\zeta \in U$ be arbitrary. By Proposition \ref{prop-zeta-b} we have
\begin{align*}
\bbp_{\zeta} \circ \bigg( \frac{\langle b, \widehat{\zeta}(Y) \rangle - \langle b,\zeta \rangle}{\sigma \sqrt{\langle Q b, \Pi_U b \rangle}} \bigg)^{-1} = {\rm N}(0,1).
\end{align*}
Let $\Phi : \bbr \to [0,1]$ be the distribution function of the standard normal distribution ${\rm N}(0,1)$, and set $c := z_{1-\frac{\alpha}{2}}$. Then we obtain
\begin{align*}
&\bbp_{\zeta} \Big( \langle b,\zeta \rangle \in \langle b,\widehat{\zeta}(Y) \rangle \pm c \sigma \sqrt{\langle Q b, \Pi_U b \rangle} \Big)
= \bbp_{\zeta} \bigg( \bigg| \frac{\langle b, \widehat{\zeta}(Y) \rangle - \langle b,\zeta \rangle}{\sigma \sqrt{\langle Q b, \Pi_U b \rangle}} \bigg| \leq c \bigg)
\\ &= \Phi(c) - \Phi(-c) = 2 \Phi(c) - 1 = 2 \bigg( 1 - \frac{\alpha}{2} \bigg) - 1 = 1 - \alpha,
\end{align*}
completing the proof.
\end{proof}

\begin{remark}
If $Q > 0$, then by Lemma \ref{lemma-Q-Pi} we have $\langle Q b, \Pi_U b \rangle > 0$  for all $b \in H \setminus U^{\perp}$.
\end{remark}

\begin{example}\label{ex-Wiener-conf}
Consider the statistical model from Example \ref{ex-Wiener-amplitude}, where we have estimated the amplitude of a noisy oscillation. Taking into account (\ref{Wiener-eigenvalues}) and (\ref{Wiener-eigenvectors}) and noting that $b = \sqrt{2} e_k$, we obtain
\begin{align*}
\la Qb,\Pi_U b \ra = \la Qb,b \ra = 2 \la Qe_k,e_k \ra = 2 \lambda_k = \frac{2}{(k - \frac{1}{2})^2 \pi^2},
\end{align*}
which provides the confidence interval (\ref{conf-int-1}) constructed in Theorem \ref{thm-conf-known}.
\end{example}

\section{Estimation of the variance}\label{sec-est-var}

In this section we estimate the variance $\sigma^2$ of the statistical model (\ref{model}). For this purpose, we are looking for quadratic estimators of the type $\| S(Y) \|^2$ with a linear operator $S \in L(H)$. The following auxiliary result is a consequence of Proposition \ref{prop-Gauss-rv}.

\begin{lemma}\label{lemma-E-square}
Let $S \in L(H)$ be a linear operator with $U \subset \ker(S)$. Then for all $\vartheta = (\zeta,\sigma^2) \in \Theta$ we have
\begin{align*}
\bbe_{\vartheta} \big[ \| S(Y) \|^2 \big] = \sigma^2 \cdot \tr (S Q S^*).
\end{align*}
\end{lemma}

Together with Lemma \ref{lemma-independence-ker}, Lemma \ref{lemma-risk-decomposition} and Corollary \ref{cor-expectation-norm-square} we arrive at the following result.

\begin{proposition}
Let $T \in L(H)$ be a linear operator with $U \subset \ker(T)$ and $\tr(T Q T^*) > 0$. We define $S \in L(H)$ as
\begin{align}\label{S-def}
S := \frac{T}{\sqrt{\tr(T Q T^*)}}.
\end{align}
Then the following statements are true:
\begin{enumerate}
\item[(a)] The estimator $\| S(Y) \|^2$ is an unbiased estimator for $\sigma^2$.

\item[(b)] For each $\vartheta \in \Theta$ the estimators $\widehat{\zeta}(Y)$ and $\| S(Y) \|^2$ are $\bbp_{\vartheta}$-independent.

\item[(c)] For each $\vartheta = (\zeta,\sigma^2) \in \Theta$ we have
\begin{align*}
R_{\vartheta} \big[ \| S(Y) \|^2 \big] = 2 \bigg( \sigma^2 \frac{\| T Q T^* \|_{L_2}}{\| T Q T^* \|_{L_1}} \bigg)^2.
\end{align*}
In particular, for each $\vartheta = (\zeta,\sigma^2) \in \Theta$ we have
\begin{align*}
R_{\vartheta} \big[ \| S(Y) \|^2 \big] \leq 2 \sigma^4.
\end{align*}

\end{enumerate}
\end{proposition}

Now, assuming that $Q(\Pi_H - \Pi_U) \neq 0$ it is natural to choose the linear operator $T := \Pi_H - \Pi_U$. Let $S \in L(H)$ be given by (\ref{S-def}), and set
\begin{align*}
\widehat{s}^2(Y) := \| S(Y) \|^2.
\end{align*}
Then our estimator for the variance $\sigma^2$ has the representation
\begin{align}\label{est-var}
\widehat{s}^2(Y) = \frac{\| Y - \widehat{\zeta}(Y) \|^2}{\tr(Q (\Pi_H - \Pi_U))}. 
\end{align}

\begin{example}
Consider the statistical model from Example \ref{ex-Wiener-amplitude}, where we have estimated the amplitude of a noisy oscillation. Taking into account (\ref{Wiener-eigenvalues}) and (\ref{Wiener-trace}), the estimator (\ref{est-var}) for the variance is given by
\begin{align*}
\widehat{s}^2(Y) = \bigg( \frac{1}{2} - \frac{1}{(k-\frac{1}{2})^2 \pi^2} \bigg)^{-1} \| Y - \widehat{\zeta}(Y) \|^2.
\end{align*}
\end{example}

\section{Confidence intervals with unknown variance}\label{sec-conf-unknown}

In this section we construct confidence intervals for the mean $\zeta$ using the statistical model (\ref{model}), where the variance $\sigma^2$ is unknown. We assume $Q(\Pi_H - \Pi_U) \neq 0$ and introduce the quantities $\tau,\lambda > 0$ and $n \in \bbn$ as
\begin{align}\label{conf-def-1}
\tau &:= \tr(Q(\Pi_H - \Pi_U)),
\\ \label{conf-def-2} \lambda &:= \| Q(\Pi_H - \Pi_U) \|,
\\ \label{conf-def-3} n &:= \dim \ker ( \lambda - Q(\Pi_H - \Pi_U) ).
\end{align}
For $a \in (0,1)$ we denote by $t_{n, a}$ the $a$-quantile of the Student's $t$-distribution $t_n$.

\begin{theorem}\label{thm-conf-unknown}
Let $b \in H$ with $\langle Q b,\Pi_U b \rangle > 0$ and $\alpha \in (0,1)$ be arbitrary. Then the random interval
\begin{align}\label{conf-interval}
\langle b,\widehat{\zeta}(Y) \rangle \pm \sqrt{\frac{\tau}{\lambda n}} \, t_{n, 1-\frac{\alpha}{2}} \, \widehat{s}(Y) \sqrt{\langle Q b,\Pi_U b \rangle}
\end{align}
is an $(1-\alpha)$-confidence interval for $\langle b,\zeta \rangle$.
\end{theorem}

\begin{proof}
Let $\vartheta = (\zeta, \sigma) \in \Theta$ be arbitrary. By Proposition \ref{prop-zeta-b} we have we have $\bbp_{\vartheta} \circ Z^{-1} = {\rm N}(0,1)$, where $Z$ denotes the random variable
\begin{align*}
Z := \frac{\langle b,\widehat{\zeta}(Y) \rangle - \langle b,\zeta \rangle}{\sigma \sqrt{\langle Q b, \Pi_U b \rangle}}.
\end{align*}
Taking into account Remark \ref{rem-pivot}, by Proposition \ref{prop-est-dist-subspaces} there exists a linear operator $S \in L(H)$ with $U \subset \ker(S)$ such that the following conditions are satisfied:
\begin{itemize}
\item $\widehat{\zeta}((Y-\zeta)/\sigma)$ and $S(Y/\sigma)$ are $\bbp_{\vartheta}$-independent.

\item We have $\| S(Y / \sigma) \| \leq \| \Pi_{U^{\perp}} (Y / \sigma) \|$.

\item We have $\bbp_{\vartheta} \circ ( \| S(Y/\sigma) \|^2 )^{-1} = \Gamma(\frac{n}{2},\frac{1}{2 \lambda})$.
\end{itemize}
We define $\widehat{t}^2 : H \to \bbr_+$ as
\begin{align*}
\widehat{t}^2(y) := \frac{\| Sy \|^2}{\tau}, \quad y \in H.
\end{align*}
Then we have $\bbp_{\vartheta} \circ ( \widehat{t}^2(Y / \sigma) )^{-1} = \Gamma(\frac{n}{2},\frac{\tau}{2 \lambda})$ and
\begin{align}\label{est-s-t}
\frac{\widehat{s}^2(Y)}{\sigma^2} = \widehat{s}^2 \bigg( \frac{Y}{\sigma} \bigg) = \frac{1}{\tau} \| \Pi_{U^{\perp}}(Y / \sigma) \|^2 \geq \frac{1}{\tau} \| S(Y / \sigma) \|^2 = \widehat{t}^2 \bigg( \frac{Y}{\sigma} \bigg) = \frac{\widehat{t}^2(Y)}{\sigma^2}.
\end{align}
By Proposition \ref{prop-Pearson} there is a random variable $W$ with $\bbp_{\vartheta} \circ W^{-1} = t_n$ such that
\begin{align*}
\frac{\langle b,\widehat{\zeta}(Y) \rangle - \langle b,\zeta \rangle}{\widehat{t}(Y) \sqrt{\langle Q b, \Pi_U b \rangle}} = \frac{Z}{\sqrt{\widehat{t}^2(Y / \sigma)}} = \sqrt{\frac{\tau}{\lambda n}} W.
\end{align*}
We denote by $F : \bbr \to [0,1]$ the distribution function of the Student's $t$-distribution $t_n$. Furthermore, we set 
\begin{align*}
c := \sqrt{\frac{\tau}{\lambda n}} \cdot t_{n, 1-\frac{\alpha}{2}} \quad \text{and} \quad d := t_{n, 1-\frac{\alpha}{2}}.
\end{align*}
Since the distribution $t_n$ has a symmetric density, by (\ref{est-s-t}) we obtain
\begin{align*}
&\bbp_{\vartheta} \Big( \langle b,\zeta \rangle \in \langle b,\widehat{\zeta}(Y) \rangle \pm c \, \widehat{s}(Y) \sqrt{\langle Q b,\Pi_U b \rangle} \Big) 
= \bbp_{\vartheta} \bigg( \bigg| \frac{\langle b,\widehat{\zeta}(Y) \rangle - \langle b,\zeta \rangle}{\widehat{s}(Y) \sqrt{\langle Q b,\Pi_U b \rangle}} \bigg| \leq c \bigg) 
\\ &\geq \bbp_{\vartheta} \bigg( \bigg| \frac{\langle b,\widehat{\zeta}(Y) \rangle - \langle b,\zeta \rangle}{\widehat{t}(Y) \sqrt{\langle Q b,\Pi_U b \rangle}} \bigg| \leq c \bigg) = \bbp_{\vartheta} \bigg( \bigg| \sqrt{\frac{ \tau}{\lambda n}} W \bigg| \leq c \bigg) = \bbp_{\vartheta} ( |W| \leq d )
\\ &= F(d) - F(-d) = 2 F(d) - 1 = 2 \bigg( 1 - \frac{\alpha}{2} \bigg) - 1 = 1 - \alpha,
\end{align*}
completing the proof.
\end{proof}

\begin{example}
Consider the statistical model from Example \ref{ex-Wiener-amplitude}, where we have estimated the amplitude of a noisy oscillation. Here we assume that $k \geq 2$. Then, noting (\ref{Wiener-eigenvalues}) and (\ref{Wiener-trace}) we obtain
\begin{align*}
\tau &= \frac{1}{2} - \frac{1}{(k-\frac{1}{2})^2 \pi^2},
\\ \lambda &= \frac{4}{\pi^2},
\\ n &= 1.
\end{align*}
Furthermore, in Example \ref{ex-Wiener-conf} we have calculated
\begin{align*}
\la Qb,\Pi_U b \ra = \frac{2}{(k-\frac{1}{2})^2 \pi^2},
\end{align*}
which provides the confidence interval (\ref{conf-interval}) constructed in Theorem \ref{thm-conf-unknown}.
\end{example}

\begin{remark}
Consider the well-known finite dimensional situation $H = \bbr^d$ and $Q = \Id$. With $e := \dim(U)$ the quantities (\ref{conf-def-1})--(\ref{conf-def-3}) are given by $\tau = d-e$, $\lambda = 1$ and $n = d-e$. Therefore, in this particular situation the confidence interval (\ref{conf-interval}) is given by the well-known expression
\begin{align*}
\langle b,\widehat{\zeta}(Y) \rangle \pm t_{d-e, 1-\frac{\alpha}{2}} \, \widehat{s}(Y) \| \Pi_U b \|.
\end{align*}
\end{remark}

\section{Hypothesis testing}\label{sec-test}

In this section we perform hypothesis testing for the test problem that the mean $\zeta$ belongs to a smaller subspace. More precisely, let $U_0 \subset U$ be a closed subspace with $U_0 \neq \{ 0 \}$ and $U_0 \neq U$. Consider the test problem
\begin{align}\label{test-problem}
H_0 : \{ \zeta \in U_0 \} \quad \text{versus} \quad H_1 : \{ \zeta \notin U_0 \}.
\end{align}
Introducing $\Theta_0 := U_0 \times (0,\infty)$, the test problem can be written as
\begin{align*}
H_0 : \{ \vartheta \in \Theta_0 \} \quad \text{versus} \quad H_1 : \{ \vartheta \notin \Theta_0 \}.
\end{align*}
We assume that $U_0$ and $U_0^{\perp} \cap U$ are $Q$-invariant. Furthermore, we assume that $Q(\Pi_H - \Pi_U), Q(\Pi_U - \Pi_{U_0}) \neq 0$, and that the range of $Q(\Pi_U - \Pi_{U_0})$ is finite dimensional. We define $\lambda,\mu > 0$ and $n,m \in \bbn$ as
\begin{align}\label{test-def-1}
\lambda &:= \| Q (\Pi_H - \Pi_U) \|,
\\ \label{test-def-2} \mu &:= \| Q (\Pi_U - \Pi_{U_0}) \|,
\\ \label{test-def-3} n &:= \dim \ker ( \lambda - Q (\Pi_H - \Pi_U) ),
\\ \label{test-def-4} m &:= \dim \ran ( Q(\Pi_U - \Pi_{U_0}) ).
\end{align}
For $a \in (0,1)$ we denote by ${\rm F}_{m,n,a}$ the $a$-quantile of the Fisher distribution ${\rm F}_{m,n}$.

\begin{theorem}\label{thm-test}
For each $\alpha \in (0,1)$ the statistical test
\begin{align}\label{test}
\delta(Y) := \bbI \bigg\{ \frac{n \lambda}{m \mu} \frac{\| \Pi_U(Y) - \Pi_{U_0}(Y) \|^2}{\| \Pi_H(Y) - \Pi_U(Y) \|^2} \geq {\rm F}_{m,n,1-\alpha} \bigg\}
\end{align}
for the test problem (\ref{test-problem}) has level of significance $\alpha$.
\end{theorem}

\begin{proof}
Let $\vartheta = (\zeta,\sigma^2) \in \Theta_0$ be arbitrary. Taking into account Remark \ref{rem-pivot}, by Proposition \ref{prop-est-dist-subspaces} there exist linear operators $S,T \in L(H)$ with $U \subset \ker(S)$ and $U_0 \oplus U^{\perp} \subset \ker(T)$ such that the following conditions are satisfied:
\begin{itemize}
\item $S(Y/\sigma)$ and $T(Y/\sigma)$ are $\bbp_{\vartheta}$-independent.

\item We have $\| S(Y/\sigma) \| \leq \| \Pi_{U^{\perp}}(Y / \sigma) \|$ and $\| \Pi_{U_0^{\perp} \cap U}(Y/\sigma) \| \leq \| T(Y / \sigma) \|$.

\item  We have the distributions
\begin{align*}
\bbp_{\vartheta} \circ \big( \| S(Y/\sigma) \|^2 \big)^{-1} = \Gamma \bigg( \frac{n}{2},\frac{1}{2 \lambda} \bigg) \quad \text{and} \quad \bbp_{\vartheta} \circ \big( \| T(Y/\sigma) \|^2 \big)^{-1} = \Gamma \bigg( \frac{m}{2},\frac{1}{2 \mu} \bigg).
\end{align*}
\end{itemize}
Therefore, we obtain
\begin{align*}
\frac{\| \Pi_U(Y) - \Pi_{U_0}(Y) \|}{\| \Pi_H(Y) - \Pi_U(Y) \|} = \frac{\| \Pi_{U_0^{\perp} \cap U}(Y / \sigma) \|}{\| \Pi_{U^{\perp}}(Y / \sigma) \|} \leq \frac{\| T(Y / \sigma) \|}{\| S(Y / \sigma) \|}.
\end{align*}
Furthermore, by Proposition \ref{prop-F-distribution} there is a random variable $Z$ with $\bbp_{\vartheta} \circ Z^{-1} = {\rm F}_{m,n}$ such that
\begin{align*}
\frac{\| T(Y / \sigma) \|^2}{\| S(Y / \sigma) \|^2} = \frac{m \mu}{n \lambda} Z.
\end{align*}
Therefore, we obtain
\begin{align*}
\bbe_{\vartheta}[\delta(Y)] &= \bbp_{\vartheta} \bigg( \frac{n \lambda}{m \mu} \frac{\| \Pi_U(Y) - \Pi_{U_0}(Y) \|^2}{\| \Pi_H(Y) - \Pi_U(Y) \|^2} \geq {\rm F}_{m,n,1-\alpha} \bigg)
\\ &\leq \bbp_{\vartheta} \bigg( \frac{n \lambda}{m \mu} \frac{\| T(Y/\sigma) \|^2}{\| S(Y/\sigma) \|^2} \geq {\rm F}_{m,n,1-\alpha} \bigg)
\\ &= \bbp_{\vartheta} ( Y \geq {\rm F}_{m,n,1-\alpha} ) = \alpha,
\end{align*}
completing the proof.
\end{proof}

\begin{remark}
In view of the quotient appearing in (\ref{test}), note the Pythagorean identity
\begin{align*}
\| \Pi_H(Y) - \Pi_{U_0}(Y) \|^2 = \| \Pi_H(Y) - \Pi_U(Y) \|^2 + \| \Pi_U(Y) - \Pi_{U_0}(Y) \|^2.
\end{align*}
This is a consequence of Proposition \ref{prop-dist-proj}.
\end{remark}

\begin{example}
Consider the statistical model from Example \ref{ex-Wiener-amplitude}, where we have estimated the amplitude of a noisy oscillation. Slightly different than in the previous examples, we assume that $U = \lin \{ e_k,e_{k+1},e_{k+2} \}$ and $U_0 = \lin \{ e_k \}$ for some $k \geq 2$. Then, noting (\ref{Wiener-eigenvalues}) we obtain
\begin{align*}
\lambda &= \frac{4}{\pi^2},
\\ \mu &= \frac{1}{(k+\frac{1}{2})^2 \pi^2},
\\ n &= 1,
\\ m &= 2,
\end{align*}
which provides the statistical test (\ref{test}) constructed in Theorem \ref{thm-test}.
\end{example}

\begin{remark}
Consider the well-known finite dimensional situation $H = \bbr^d$ and $Q = \Id$. With $e := \dim(U)$ and $f := \dim(U_0)$ the quantities (\ref{test-def-1})--(\ref{test-def-4}) are given by $\lambda = \mu = 1$, $n = d-e$ and $m = e-f$. Therefore, in this particular situation the statistical test (\ref{test}) is given by the well-known expression
\begin{align*}
\delta(Y) = \bbI \bigg\{ \frac{d-e}{e-f} \frac{\| \Pi_U(Y) - \Pi_{U_0}(Y) \|^2}{\| \Pi_H(Y) - \Pi_U(Y) \|^2} \geq {\rm F}_{e-f,d-e,1-\alpha} \bigg\}.
\end{align*}
\end{remark}

\section{Linear regression in infinite dimensions}\label{sec-learning}

In this section we treat a linear regression problem in infinite dimensions, which may also have applications to statistical and machine learning. Let $H$ and $G$ be separable Hilbert spaces, and let $Q \in L_1^+(H)$ be a nonnegative self-adjoint nuclear operator. We consider the statistical model
\begin{align*}
\calm = \big( H, \calb(H), (\bbq_{\gamma} : \gamma \in \Gamma) \big)
\end{align*}
where $\Gamma := G \times (0,\infty)$ and
\begin{align*}
\bbq_{\gamma} = \N(A \beta, \sigma^2 Q), \quad \gamma = (\beta,\sigma^2) \in \Gamma
\end{align*}
with an injective linear operator $A \in L(G,H)$. Thus we conjecture a linear dependence
\begin{align*}
Y = A \beta + \sigma \epsilon
\end{align*}
with a noise $\epsilon \sim \N(0,Q)$, where $A$ is allowed to depend on a control variable $x \in G$, which can be the realization of a $G$-valued random variable $X$. Hence, we are interested in learning $\beta$, which determines the linear dependence between $X$ and $Y$.

\begin{example}
A simple example is the linear regression
\begin{align*}
Y = \beta_0 e + \beta_1 x + \sigma \epsilon
\end{align*}
with $\beta_0, \beta_1 \in \bbr$, $e \in H$ and a control variable $x \in H$ such that $e$ and $x$ are linearly independent. In this case we have $G = \bbr^2$ and the injective linear operator $A \in L(\bbr^2,H)$ is given by $A \beta = \beta_0 e + \beta_1 x$ for each $\beta = (\beta_0,\beta_1) \in \bbr^2$.
\end{example}

An estimator of the form $\widehat{\beta}(Y)$ with a measurable mapping $\widehat{\beta} : H \to G$ is called a \emph{least square estimator (LSE)} for $\beta$ if
\begin{align*}
\| Y - A \widehat{\beta}(Y) \| = \min_{\beta \in G} \| Y - A \beta \|.
\end{align*}
We set $U := \ran(A)$. Then the linear operator $A$ is an isomorphism $A \in L(G,U)$ due to the open mapping theorem, and hence $U$ is a closed subspace of $H$. For what follows, we assume that $U$ and $U^{\perp}$ are $Q$-invariant.

Note that the statistical model $\calm$ coincides with (\ref{model}), where $\Theta := U \times (0,\infty)$ and the family $(\bbp_{\vartheta} : \vartheta \in \Theta)$ is given by (\ref{fam-measures}). More precisely, for each $\vartheta = (\zeta,\sigma^2) \in \Theta$ we have $\bbp_{\vartheta} = \bbq_{\gamma}$, where $\gamma = (\beta,\sigma^2) \in \Gamma$ is given by $\beta = A^{-1} \zeta$. Conversely, for each $\gamma = (\beta,\sigma^2) \in \Gamma$ we have $\bbq_{\gamma} = \bbp_{\vartheta}$, where $\vartheta = (\zeta,\sigma^2) \in \Theta$ is given by $\zeta = A \beta$. 

\begin{proposition}
There is a unique LSE for $\beta$; it is given by $\widehat{\beta}(Y) = A^{-1} \widehat{\zeta}(Y)$.
\end{proposition}

\begin{proof}
Note that an estimator $\widehat{\beta}(Y)$ is a LSE for $\beta$ if and only if
\begin{align*}
A \widehat{\beta}(Y) = \Pi_U Y,
\end{align*}
which is equivalent to $\widehat{\beta}(Y) = A^{-1} \Pi_U Y = A^{-1} \widehat{\zeta}(Y)$.
\end{proof}

As a consequence, we have $\widehat{\zeta}(Y) = A \widehat{\beta}(Y)$. These findings allow us to transfer our previous results to the present setting. Let us first consider the statistical model (\ref{model-0}), where the probability measures are given by (\ref{family-0}); that is, the variance $\sigma^2$ is assumed to be known. In view of the upcoming results, note that $(A^{-1})^* = (A^*)^{-1}$; see, for example \cite[Lemma 6.1]{Tappe-cones-2}.

\begin{corollary}
Let $c \in H$ with $\langle Q (A^{-1})^* c, \Pi_U (A^{-1})^* c \rangle > 0$ and $\alpha \in (0,1)$ be arbitrary. Then the random interval
\begin{align*}
\langle c,\widehat{\beta}(Y) \rangle \pm z_{1-\frac{\alpha}{2}} \, \sigma \sqrt{\langle Q (A^{-1})^* c, \Pi_U (A^{-1})^* c \rangle}
\end{align*}
is an $(1-\alpha)$-confidence interval for $\langle c,\beta \rangle$.
\end{corollary}

\begin{proof}
For each $\beta \in G$ we have
\begin{align*}
&\left\{ \la c,\beta \ra \in \langle c,\widehat{\beta}(Y) \rangle \pm z_{1-\frac{\alpha}{2}} \, \sigma \sqrt{\langle Q (A^{-1})^* c, \Pi_U (A^{-1})^* c \rangle} \right\}
\\ &= \left\{ \la c,A^{-1} \zeta \ra \in \langle c,A^{-1} \widehat{\zeta}(Y) \rangle \pm z_{1-\frac{\alpha}{2}} \, \sigma \sqrt{\langle Q (A^{-1})^* c, \Pi_U (A^{-1})^* c \rangle} \right\}
\\ &= \left\{ \la (A^{-1})^* c,\zeta \ra \in \langle (A^{-1})^* c, \widehat{\zeta}(Y) \rangle \pm z_{1-\frac{\alpha}{2}} \, \sigma \sqrt{\langle Q (A^{-1})^* c, \Pi_U (A^{-1})^* c \rangle} \right\}.
\end{align*}
Therefore, applying Theorem \ref{thm-conf-known} completes the proof.
\end{proof}

Now consider the statistical model (\ref{model}), where the parameter domain is given by $\Theta := U \times (0,\infty)$, and where the probability measures are the Gaussian measures given by (\ref{fam-measures}). We assume $Q(\Pi_H - \Pi_U) \neq 0$ and introduce $\tau,\lambda > 0$ and $n \in \bbn$ by (\ref{conf-def-1})--(\ref{conf-def-3}).

\begin{corollary}
Let $b \in H$ with $\langle Q (A^{-1})^* c, \Pi_U (A^{-1})^* c \rangle > 0$ and $\alpha \in (0,1)$ be arbitrary. Then the random interval
\begin{align*}
\langle c,\widehat{\beta}(Y) \rangle \pm \sqrt{\frac{\tau}{\lambda n}} \, t_{n, 1-\frac{\alpha}{2}} \, \widehat{s}(Y) \sqrt{\langle Q (A^{-1})^* c,\Pi_U (A^{-1})^* c \rangle}
\end{align*}
is an $(1-\alpha)$-confidence interval for $\langle c,\beta \rangle$.
\end{corollary}

\begin{proof}
For each $\beta \in G$ we have
\begin{align*}
&\bigg\{ \la c,\beta \ra \in \langle c,\widehat{\beta}(Y) \rangle \pm \sqrt{\frac{\tau}{\lambda n}} \, t_{n, 1-\frac{\alpha}{2}} \, \widehat{s}(Y) \sqrt{\langle Q (A^{-1})^* c,\Pi_U (A^{-1})^* c \rangle} \bigg\}
\\ &= \bigg\{ \la c, A^{-1} \zeta \ra \in \langle c,A^{-1} \widehat{\zeta}(Y) \rangle \pm \sqrt{\frac{\tau}{\lambda n}} \, t_{n, 1-\frac{\alpha}{2}} \, \widehat{s}(Y) \sqrt{\langle Q (A^{-1})^* c,\Pi_U (A^{-1})^* c \rangle} \bigg\}
\\ &= \bigg\{ \la (A^{-1})^* c, \zeta \ra \in \langle (A^{-1})^* c, \widehat{\zeta}(Y) \rangle
\\ &\qquad \qquad \qquad \qquad \quad \pm \sqrt{\frac{\tau}{\lambda n}} \, t_{n, 1-\frac{\alpha}{2}} \, \widehat{s}(Y) \sqrt{\langle Q (A^{-1})^* c,\Pi_U (A^{-1})^* c \rangle} \bigg\}.
\end{align*}
Therefore, applying Theorem \ref{thm-conf-unknown} completes the proof.
\end{proof}

Now, let $G_0 \subset G$ be a closed subspace with $G_0 \neq \{ 0 \}$ and $G_0 \neq G$. Consider the test problem
\begin{align}\label{test-2}
H_0 : \{ \beta \in G_0 \} \quad \text{versus} \quad H_1 : \{ \zeta \notin G_0 \}.
\end{align}
Introducing $\Gamma_0 := G_0 \times (0,\infty)$, the test problem can be written as
\begin{align*}
H_0 : \{ \gamma \in \Gamma_0 \} \quad \text{versus} \quad H_1 : \{ \gamma \notin \Gamma_0 \}.
\end{align*}
Set $U_0 := \ran(G_0)$. We assume that $U_0$ and $U_0^{\perp} \cap U$ are $Q$-invariant. Furthermore, we assume that $Q(\Pi_H - \Pi_U), Q(\Pi_U - \Pi_{U_0}) \neq 0$, and that the range of $Q(\Pi_U - \Pi_{U_0})$ is finite dimensional. We define $\lambda,\mu > 0$ and $n,m \in \bbn$ as (\ref{test-def-1})--(\ref{test-def-4}). As an immediate consequence of Theorem \ref{thm-test} we obtain the following result.

\begin{corollary}
For each $\alpha \in (0,1)$ the statistical test (\ref{test}) for the test problem (\ref{test-2}) has level of significance $\alpha$.
\end{corollary}

\begin{appendix}

\section{Linear operators in Hilbert spaces}\label{app-operators}

In this appendix we provide the required results about linear operators in Hilbert spaces. For two Hilbert spaces $H$ and $G$ the notations $L(H,G)$, $L_1(H,G)$, $L_2(H,G)$ and $L^+(H,G)$ mean the sets of all bounded linear operators, nuclear operators, Hilbert-Schmidt operators and nonnegative operators. If $G=H$, then we also use the notations $L(H)$, $L_1(H)$, $L_2(H)$ and $L^+(H)$. For what follows, let $H$ be a Hilbert space. For a closed subspace $U \subset H$ we denote by $\Pi_U \in L(H)$ the orthogonal projection on $U$. In particular, $\Pi_H$ is the identity operator on $H$. Given a linear operator $T \in L(H)$, a subspace $U \subset H$ is called \emph{$T$-invariant} if $T(U) \subset U$.

\begin{lemma}\label{lemma-self-adjoint}
Let $T \in L(H)$ and a subspace $U \subset H$ be such that $T|_U = \Id|_U$. Then $U^{\perp}$ is $T^*$-invariant.
\end{lemma}

\begin{proof}
Let $y \in U^{\perp}$ be arbitrary. Then we have
\begin{align*}
\la x,T^* y \ra = \la Tx,y \ra = \la x,y \ra \quad \text{for all $x \in U$.}
\end{align*}
Hence we obtain
\begin{align*}
\la T^* y - y, x \ra = 0 \quad \text{for all $x \in U$.}
\end{align*}
It follows that $T^* y - y \in U^{\perp}$, and thus $T^* y \in U^{\perp}$.
\end{proof}

\begin{lemma}\label{lemma-operator-kernel}
Let $T \geq 0$ be a self-adjoint, compact operator. Then for each $x \in H$ with $\la Tx,x \ra = 0$ we have $x \in \ker(T)$.
\end{lemma}

\begin{proof}
Since $T^{1/2}$ is also self-adjoint, we have
\begin{align*}
0 = \la Tx,x \ra = \la T^{1/2} T^{1/2} x,x \ra = \la T^{1/2} x, T^{1/2} x \ra = \| T^{1/2} x \|^2,
\end{align*}
and hence $T^{1/2} x = 0$. Therefore, we obtain $Tx =  T^{1/2} T^{1/2} x = 0$.
\end{proof}

The following auxiliary result is obvious.

\begin{lemma}\label{lemma-Q-Pi}
Let $Q \in L^+(H)$ be a nonnegative operator, and let $U \subset H$ be a closed subspace such that $U$ and $U^{\perp}$ are $Q$-invariant.
\begin{enumerate}
\item[(a)] We have $\Pi_U Q \Pi_U = Q \Pi_U = \Pi_U Q$.

\item[(b)] For each $b \in H$ we have
\begin{align*}
\la Q \Pi_U b, \Pi_U b \ra &= \la Q b, \Pi_U b \ra = \la Q \Pi_U b, b \ra,
\\ \la Q b, b \ra &= \la Q \Pi_U b, \Pi_U b \ra + \la Q \Pi_{U^{\perp}} b, \Pi_{U^{\perp}} b \ra.
\end{align*}
In particular, we have
\begin{align*}
\la Q \Pi_U b, \Pi_U b \ra \leq \la Q b, b \ra.
\end{align*}
\end{enumerate}
\end{lemma}

\begin{lemma}\label{lemma-proj-commute}
Let $U, V \subset H$ be two closed subspaces such that $\Pi_U \Pi_V = \Pi_V \Pi_U$. Then we have $\Pi_U \Pi_V = \Pi_{U \cap  V}$.
\end{lemma}

\begin{proof}
The linear operator $T := \Pi_U \Pi_V$ is a self-adjoint projection. Hence, by \cite[Thm. 12.14]{Rudin} the operator $T$ is an orthogonal projection; that is, $\ran(T) = \ker(T)^{\perp}$. We claim that $\ran(T) = U \cap V$. Indeed, for each $x \in U \cap V$ we have $Tx = \Pi_V \Pi_U x = \Pi_V x = x$, showing that $x \in \ran(T)$. Conversely, for each $y \in \ran(T)$ we have $Tx = y$ for some $x \in H$. Therefore, we have $\Pi_U \Pi_V x = y$, showing that $y \in U$. Moreover, we have $\Pi_V \Pi_U x = y$, showing that $y \in V$.
\end{proof}

\begin{lemma}\label{lemma-Pi-U-V}
Let $U \subset V \subset H$ be two closed subspaces. Then we have
\begin{align*}
\Pi_V - \Pi_U = \Pi_{U^{\perp} \cap V}.
\end{align*}
\end{lemma}

\begin{proof}
We have
\begin{align*}
\Pi_V - \Pi_U = \Pi_V - \Pi_V \Pi_U = \Pi_V (\Id - \Pi_U) = \Pi_V \Pi_{U^{\perp}}
\end{align*}
as well as
\begin{align*}
\Pi_V - \Pi_U = \Pi_V - \Pi_U \Pi_V = (\Id - \Pi_U) \Pi_V = \Pi_{U^{\perp}} \Pi_V.
\end{align*}
Therefore, applying Lemma \ref{lemma-proj-commute} completes the proof.
\end{proof}

\begin{proposition}\label{prop-dist-proj}
Let $U \subset V \subset W \subset H$ be three closed subspaces. Then for each $x \in H$ we have
\begin{align*}
\| \Pi_{W} x - \Pi_{U} x \|^2 = \| \Pi_{W} x - \Pi_{V} x \|^2 + \| \Pi_{V} x - \Pi_{U} x \|^2.
\end{align*}
\end{proposition}

\begin{proof}
Using Lemma \ref{lemma-Pi-U-V} we obtain
\begin{align*}
\| \Pi_{W} x - \Pi_{U} x \|^2 &= \| (\Pi_{W} x - \Pi_{V} x) + (\Pi_{V} x - \Pi_{U} x) \|^2
\\ &= \| \Pi_{W \cap V^{\perp}} x + \Pi_{V \cap U^{\perp}} x \|^2
\\ &= \| \Pi_{W \cap V^{\perp}} x \|^2 + \| \Pi_{V \cap U^{\perp}} x \|^2
\\ &= \| \Pi_{W} x - \Pi_{V} x \|^2 + \| \Pi_{V} x - \Pi_{U} x \|^2,
\end{align*}
completing the proof.
\end{proof}

\section{Gaussian random variables in Hilbert spaces}\label{app-Gaussian}

In this appendix we provide the required results about Gaussian random variables in Hilbert spaces. For more details we refer, for example, to \cite[Sec. 2.1.1]{Atma-book}, \cite[Sec. 2.1]{Prevot-Roeckner}, \cite[Sec. 2.3.1]{Da_Prato} or \cite[Sec. 2.1]{Liu-Roeckner}.

Let $(\Omega,\calf,\bbp)$ be a probability space, and let $H$ be a separable Hilbert space. We recall that a random variable $Y : \Omega \to H$ is called a \emph{Gaussian random variable} if for each $b \in H$ the real-valued random variable $\langle b,Y \rangle$ has a (possibly degenerated) normal distribution. For a Gaussian random variable $Y$ there are a unique element $\zeta \in H$ and a unique nonnegative self-adjoint nuclear operator $Q \in L_1^+(H)$ such that
\begin{align*}
\bbe[\langle b,Y \rangle] &= \langle b,\zeta \rangle \quad \text{for all $b \in H$,}
\\ {\rm Cov}(\langle b,Y \rangle,\langle c,Y \rangle) &= \langle Q b,c \rangle \quad \text{for all $b,c \in H$.}
\end{align*}
We call $\zeta$ the \emph{mean} and $Q$ the \emph{covariance operator} of $Y$, and we write $Y \sim {\rm N}(\zeta,Q)$. For what follows, let $Y \sim \N(\zeta,Q)$ be a Gaussian random variable. The following result lists some well-known properties.

\begin{proposition}\label{prop-Gauss-rv}
The following statements are true:
\begin{enumerate}
\item[(a)] For each $b \in H$ we have $\langle b,Y \rangle \sim {\rm N}(\langle b,\zeta \rangle, \langle Q b,b \rangle)$.

\item[(b)] We have $\bbe[Y] = \zeta$.

\item[(c)] We have $\bbe[\| Y - \zeta \|^2] = \tr(Q)$.

\item[(d)] Let $G$ be another separable Hilbert space, and let $T \in L(H,G)$ be a linear operator. Then we have $T(Y) \sim \N(T \zeta, T Q T^*)$.
\end{enumerate}
\end{proposition}

Together with Lemma \ref{lemma-Q-Pi} the following result is an immediate consequence.

\begin{corollary}\label{cor-proj-Gauss}
Let $U \subset H$ be closed subspace with $\zeta \in U$ such that $U$ and $U^{\perp}$ are $Q$-invariant. Then the following statements are true:
\begin{enumerate}
\item[(a)] We have $\Pi_U Y \sim \N(\zeta,Q \Pi_U)$.

\item[(b)] For each $b \in H$ we have $\la b,\Pi_U Y \ra \sim \N(\la b,\zeta \ra, \la Qb,\Pi_U b \ra)$.
\end{enumerate}
\end{corollary}

The next result is a consequence of \cite[Prop. 2.19]{Da_Prato}.

\begin{proposition}\label{prop-Gauss-moments}
We have $\bbe[ \| Y \|^m ] < \infty$ for each $m \in \bbn$.
\end{proposition}

\begin{remark}\label{rem-spectral}
By the spectral theorem for self-adjoint operators there are an at most countable index set $K$, an orthonormal system $(e_k)_{k \in K}$, and positive elements $(\lambda_k)_{k \in K} \subset (0,\infty)$ with $\sum_{k \in K} \lambda_k < \infty$ such that $H = H_1 \oplus H_2$, where $H_1 := \ker(Q)$ and $H_2 := H_1^{\perp}$ is given by
\begin{align*}
H_2 = \overline{\lin} \{ e_k : k \in K \},
\end{align*}
and we have
\begin{align}\label{Q-diagonal}
Q e_k = \lambda_k e_k \quad \text{for all $k \in K$.}
\end{align}
\end{remark}

We define the random variables $(\beta_k)_{k \in K}$ as
\begin{align*}
\beta_k := \frac{1}{\sqrt{\lambda_k}} \langle Y,e_k \rangle, \quad k \in K.
\end{align*}

\begin{lemma}\label{lemma-series-Gauss}
The following statements are true:
\begin{enumerate}
\item[(a)] The random variables $(\beta_k)_{k \in K}$ are independent and identically distributed with $\beta_k \sim {\rm N}(0,1)$ for each $k \in K$.

\item[(b)] We have the series representation
\begin{align}\label{series-Gauss}
Y = \zeta + \sum_{k \in K} \sqrt{\lambda_k} \beta_k e_k.
\end{align}
\end{enumerate}
\end{lemma}

\begin{proof}
This follows from \cite[Prop. 2.1.6]{Liu-Roeckner} and its proof.
\end{proof}

\begin{lemma}\label{lemma-independence-ker}
Let $U,V \subset H$ be two closed $Q$-invariant subspaces such that $U \cap V = \{ 0 \}$. Furthermore, let $T,S \in L(H)$ be such that $U^{\perp} \subset \ker(T)$ and $V^{\perp} \subset \ker(S)$. Then the Gaussian random variables $T(Y)$ and $S(Y)$ are independent.
\end{lemma}

\begin{proof}
Applying the spectral theorem to the operators $Q|_U$ and $Q|_V$ and recalling (\ref{Q-diagonal}), we may assume there are disjoint index sets $I,J \subset K$ such that $(e_i)_{i \in I}$ is an orthonormal system in $U$ and $(e_j)_{j \in J}$ is an orthonormal system in $V$. Using the series representation (\ref{series-Gauss}) from Lemma \ref{lemma-series-Gauss} we obtain
\begin{align*}
T(Y) &= T(\zeta) + T \bigg( \sum_{i \in I} \sqrt{\lambda_i} \beta_i e_i \bigg),
\\ S(Y) &= S(\zeta) + S \bigg( \sum_{j \in J} \sqrt{\lambda_j} \beta_j e_j \bigg),
\end{align*}
proving the Independence of $T(Y)$ and $S(Y)$.
\end{proof}

In the finite dimensional case $H = \bbr^n$ and $Q = \Id$ we have $\| Y \|^2 \sim \chi_n^2(\| \zeta \|^2)$, where $\chi_n^2(\lambda)$ denotes the noncentral $\chi^2$-distribution with $n$ degrees of freedom and non-centrality parameter $\lambda$. Hence, we have
\begin{align*}
\bbe \big[ \| Y \|^2 \big] &= n + \| \zeta \|^2,
\\ \Var \big[ \| Y \|^2 \big] &= 2 ( n + 2 \| \zeta \|^2 ).
\end{align*}
However, if $\dim H = \infty$, then the distribution of $\| Y \|^2$ is not available in closed form. However, in case $\zeta = 0$ there is an asymptotic formula for its density; see \cite{Zolotarev} and \cite{Hoeffding, Beran} for further extensions. In the situation where $\zeta = 0$ and the index set $K$ is finite, series representations for the density have been derived in \cite{Mathai} and \cite{Moschopoulos}. 

At any rate, in the present situation we can calculate expectation and variance of $\| Y \|^2$. For this purpose, we intervene with an auxiliary result.

\begin{lemma}\label{lemma-var-series}
Let $(\xi_k)_{k \in \bbn} \subset \call^2$ be a sequence of nonnegative, independent random variables $\xi_k : \Omega \to \bbr_+$ such that for the series $\xi := \sum_{k=1}^{\infty} \xi_k$ we have $\xi \in \call^2$. Then we have $\Var[\xi] = \sum_{k=1}^{\infty} \Var[\xi_k]$.
\end{lemma}

\begin{proof}
By the monotone convergence theorem we have
\begin{align*}
\bbe[\xi^2] &= \bbe \Bigg[ \bigg( \sum_{k=1}^{\infty} \xi_k \bigg)^2 \Bigg] = \bbe \Bigg[ \bigg( \lim_{n \to \infty} \sum_{k=1}^n \xi_k \bigg)^2 \Bigg] = \bbe \Bigg[ \lim_{n \to \infty} \bigg( \sum_{k=1}^n \xi_k \bigg)^2 \Bigg]
\\ &= \lim_{n \to \infty} \bbe \Bigg[ \bigg( \sum_{k=1}^n \xi_k \bigg)^2 \Bigg]
\end{align*}
as well as
\begin{align*}
\bbe[\xi]^2 &= \bbe \Bigg[ \sum_{k=1}^{\infty} \xi_k \Bigg]^2 = \bbe \Bigg[ \lim_{n \to \infty} \sum_{k=1}^n \xi_k \Bigg]^2 = \Bigg( \lim_{n \to \infty} \bbe \Bigg[ \sum_{k=1}^n \xi_k \Bigg] \Bigg)^2
\\ &= \lim_{n \to \infty} \bbe \Bigg[ \sum_{k=1}^n \xi_k \Bigg]^2.
\end{align*}
Therefore, taking into account the independence of the random variables $(\xi_k)_{k \in \bbn}$ we obtain
\begin{align*}
\Var[\xi] &= \bbe[\xi^2] - \bbe[\xi]^2 = \lim_{n \to \infty} \Bigg( \bbe \Bigg[ \bigg( \sum_{k=1}^n \xi_k \bigg)^2 \Bigg] - \bbe \Bigg[ \sum_{k=1}^n \xi_k \Bigg]^2 \Bigg)
\\ &= \lim_{n \to \infty} \Var \bigg[ \sum_{k=1}^n \xi_k \bigg] = \lim_{n \to \infty} \sum_{k=1}^n \Var[\xi_k] = \sum_{k=1}^{\infty} \Var[\xi_k],
\end{align*}
completing the proof.
\end{proof}

\begin{proposition}\label{prop-chi-square-weighted}
We have
\begin{align*}
\bbe \big[ \| Y \|^2 \big] &= \tr(Q) + \| \zeta \|^2,
\\ \Var \big[ \| Y \|^2 \big] &= 2 \big( \| Q \|_{L_2}^2 + 2 \| Q^{1/2} \zeta \|^2 \big).
\end{align*}
\end{proposition}

\begin{proof}
By the series representation (\ref{series-Gauss}) from Lemma \ref{lemma-series-Gauss} we have
\begin{align*}
Y = \zeta + \sum_{k \in K} \sqrt{\lambda_k} \beta_k e_k = \sum_{k \in K} \big( \langle \zeta,e_k \rangle + \sqrt{\lambda_k} \beta_k \big) e_k,
\end{align*}
and hence
\begin{align*}
\| Y \|^2 = \sum_{k \in K} \big| \langle \zeta,e_k \rangle + \sqrt{\lambda_k} \beta_k \big|^2 = \sum_{k \in K} \lambda_k \bigg| \frac{\langle \zeta,e_k \rangle}{\sqrt{\lambda_k}} + \beta_k \bigg|^2.
\end{align*}
For each $k \in K$ we have
\begin{align*}
\frac{\langle \zeta,e_k \rangle}{\sqrt{\lambda_k}} + \beta_k \sim \N \bigg( \frac{\langle \zeta,e_k \rangle}{\sqrt{\lambda_k}}, 1 \bigg),
\end{align*}
and hence
\begin{align*}
\bigg| \frac{\langle \zeta,e_k \rangle}{\sqrt{\lambda_k}} + \beta_k \bigg|^2 \sim \chi_1^2 \bigg( \frac{|\langle \zeta,e_k \rangle|^2}{\lambda_k} \bigg).
\end{align*}
Therefore, by the monotone convergence theorem we obtain
\begin{align*}
\bbe \big[ \| Y \|^2 \big] &= \sum_{k \in K} \lambda_k \, \bbe \Bigg[ \bigg| \frac{\langle \zeta,e_k \rangle}{\sqrt{\lambda_k}} + \beta_k \bigg|^2 \Bigg] = \sum_{k \in K} \lambda_k \bigg( 1 + \frac{|\langle \zeta,e_k \rangle|^2}{\lambda_k} \bigg)
\\ &= \sum_{k \in K} \lambda_k + \sum_{k \in K} |\langle \zeta,e_k \rangle|^2 = \tr(Q) + \| \zeta \|^2.
\end{align*}
Furthermore, by Proposition \ref{prop-Gauss-moments} and Lemma \ref{lemma-var-series} we have
\begin{align*}
\Var \big[ \| Y \|^2 \big] &= \sum_{k \in K} \lambda_k^2 \, \Var \Bigg[ \bigg| \frac{\langle \zeta,e_k \rangle}{\sqrt{\lambda_k}} + \beta_k \bigg|^2 \Bigg] = 2 \sum_{k \in K} \lambda_k^2 \bigg( 1 + 2 \frac{|\langle \zeta,e_k \rangle|^2}{\lambda_k} \bigg)
\\ &= 2 \bigg( \sum_{k \in K} \lambda_k^2 + 2 \sum_{k \in K} \lambda_k | \langle \zeta,e_k \rangle |^2 \bigg) = 2 \big( \| Q \|_{L_2}^2 + 2 \| Q^{1/2} \zeta \|^2 \big),
\end{align*}
where for the last step we note that
\begin{align*}
Q^{1/2} \zeta = \sum_{k \in K} \la \zeta,e_k \ra Q^{1/2} e_k = \sum_{k \in K} \sqrt{\lambda_k} \la \zeta,e_k \ra e_k.
\end{align*}
This completes the proof.
\end{proof}

\begin{corollary}\label{cor-expectation-norm-square}
For each $T \in L(H)$ we have
\begin{align*}
\bbe[\| T(Y) \|^2] &= \tr(T Q T^*) + \| T \zeta \|^2,
\\ \Var[\| T(Y) \|^2] &= 2 \big( \| T Q T^* \|_{L_2}^2 + 2 \| (T Q T^*)^{1/2} T \zeta \|^2 \big).
\end{align*}
\end{corollary}

\begin{proof}
This is an immediate consequence of Propositions \ref{prop-Gauss-rv} and \ref{prop-chi-square-weighted}.
\end{proof}

\begin{proposition}\label{prop-est-dist-subspaces}
Let $U \subset H$ be a closed subspace such that $U$ and $U^{\perp}$ are $Q$-invariant, and we have $\zeta \in U$ and $Q(\Pi_H - \Pi_U) \neq 0$.
\begin{enumerate}
\item[(a)] We define $\lambda > 0$ and $n \in \bbn$ as
\begin{align*}
\lambda := \| Q(\Pi_H - \Pi_U) \| \quad \text{and} \quad n := \dim \ker ( \lambda - Q (\Pi_H - \Pi_U) ).
\end{align*}
Then there exists a linear operator $S \in L(H)$ with $U \subset \ker(S)$ such that the following conditions are satisfied:
\begin{itemize}
\item $\Pi_U Y$ and $S(Y)$ are independent.

\item We have $\| S(Y) \| \leq \| Y - \Pi_U Y \|$.

\item We have $\| S(Y) \|^2 \sim \Gamma(\frac{n}{2},\frac{1}{2 \lambda})$.
\end{itemize}
\item[(b)] Let $U_0 \subset U$ be another closed subspace such that $U_0$ and $U_0^{\perp} \cap U$ are $Q$-invariant. Furthermore, we assume that $Q(\Pi_U - \Pi_{U_0}) \neq 0$ and that the range of $Q(\Pi_U - \Pi_{U_0})$ is finite dimensional. We define $\mu > 0$ and $m \in \bbn$ as
\begin{align*}
\mu := \| Q(\Pi_U - \Pi_{U_0}) \| \quad \text{and} \quad m := \dim \ran (Q(\Pi_U - \Pi_{U_0})) < \infty. 
\end{align*}
Then there exists another linear operator $T \in L(H)$ with $U_0 \oplus U^{\perp} \subset \ker(T)$ such that the following conditions are satisfied:
\begin{itemize}
\item $S(Y)$ and $T(Y)$ are independent.

\item We have $\| \Pi_U Y - \Pi_{U_0} Y \| \leq \| T(Y) \|$.

\item We have $\| T(Y) \|^2 \sim \Gamma(\frac{m}{2},\frac{1}{2 \mu})$.
\end{itemize}
\end{enumerate}
\end{proposition}

\begin{proof}
Applying the spectral theorem to the operators $Q|_U$ and $Q|_{U^{\perp}}$ and recalling (\ref{Q-diagonal}), we may assume there are disjoint index sets $I,J \subset K$ with $I \cup J = K$ such that $(e_i)_{i \in I}$ is an orthonormal system in $U$ and $(e_j)_{j \in J}$ is an orthonormal system in $U^{\perp}$.

\noindent(a) We define the $n$-dimensional subspace $V \subset U^{\perp}$ as
\begin{align*}
V := \ker ( \lambda - Q (\Pi_H - \Pi_U) ),
\end{align*}
and we define the linear operator $S \in L(H)$ as the orthogonal projection $S := \Pi_V$. Then we have $U^{\perp} = \ker(\Pi_U)$ and $U \subset V^{\perp} = \ker(S)$. Hence, by Lemma \ref{lemma-independence-ker} the random variables $\Pi_U Y$ and $S(Y)$ are independent. Furthermore, we have
\begin{align*}
\| S(Y) \| = \| \Pi_V Y \| \leq \| \Pi_{U^{\perp}} Y \| = \| Y - \Pi_U Y \|.
\end{align*}
Applying the spectral theorem to the operators $Q|_V$ and $Q|_{V^{\perp} \cap U^{\perp}}$ and recalling (\ref{Q-diagonal}), we may assume there is an index set $J_0 \subset J$ with $|J_0| = n$ such that $(e_j)_{j \in J_0}$ is an orthonormal basis of $V$. Using the series representation (\ref{series-Gauss}) from Lemma \ref{lemma-series-Gauss} we obtain
\begin{align*}
S(Y) = \sqrt{\lambda} \sum_{j \in J_0} \beta_j e_j,
\end{align*}
and hence
\begin{align*}
\| S(Y) \|^2 = \lambda \sum_{j \in J_0} \beta_j^2 \sim \Gamma \bigg( \frac{n}{2}, \frac{1}{2 \lambda} \bigg).
\end{align*}
(b) Applying the spectral theorem to the operators $Q|_{U_0}$ and $Q|_{U_0^{\perp} \cap U}$ and recalling (\ref{Q-diagonal}), we may assume there is an index set $I_0 \subset I$ with $|I_0| = m$ such that $(e_i)_{i \in I_0}$ is an orthonormal basis of $\ran (Q(\Pi_U - \Pi_{U_0}))$. We define the linear operator $T \in L(H)$ as
\begin{align*}
Ty := \sum_{i \in I_0} \sqrt{\frac{\mu}{\lambda_i}} \la y,e_i \ra e_i, \quad y \in H.
\end{align*}
Then we have $U^{\perp} \subset U_0 \oplus U^{\perp} \subset \ker(T)$. Furthermore, by part (a) we have $U \subset \ker(S)$. Hence, by Lemma \ref{lemma-independence-ker} the random variables $S(Y)$ and $T(Y)$ are independent. Using the series representation (\ref{series-Gauss}) from Lemma \ref{lemma-series-Gauss} we obtain
\begin{align*}
T(Y) = \sqrt{\mu} \sum_{i \in I_0} \beta_i e_i,
\end{align*}
and hence
\begin{align*}
\| T(Y) \|^2 = \mu \sum_{i \in I_0} \beta_i^2 \sim \Gamma \bigg( \frac{n}{2}, \frac{1}{2 \mu} \bigg).
\end{align*}
Furthermore, since $\lambda_i \leq \mu$ for all $i \in I \setminus I_0$, we have
\begin{align*}
\| \Pi_U Y - \Pi_{U_0} Y \|^2 = \| \Pi_{U_0^{\perp} \cap U} Y \|^2 = \sum_{i \in I_0} \lambda_i \beta_i^2 \leq \mu \sum_{i \in I_0} \beta_i^2 = \| T(Y) \|^2,
\end{align*}
completing the proof.
\end{proof}

\section{The Student's $t$-distribution and the Fisher distribution}\label{app-distributions}

In this appendix we provide the required results about the Student's $t$-distribution and the Fisher distribution. Let us recall that the Gamma distribution $\Gamma(\alpha,\beta)$ with parameters $\alpha,\beta \in (0,\infty)$ is the absolutely continuous probability measure on $(\bbr,\calb(\bbr))$ with density
\begin{align*}
f(x) = \frac{\beta^{\alpha}}{\Gamma(\alpha)} x^{\alpha - 1} e^{-\beta x} \bbI_{(0,\infty)}(x), \quad x \in \bbr.
\end{align*}

\begin{proposition}\label{prop-Pearson}
For two independent random variables $X \sim {\rm N}(0,1)$ and $Y \sim \Gamma(\alpha,\beta)$ we have
\begin{align*}
\frac{X}{\sqrt{Y}} = \sqrt{\frac{\beta}{\alpha}} Z
\end{align*}
with a random variable $Z \sim t_{2 \alpha}$.
\end{proposition}

\begin{proof}
We define the random variables
\begin{align*}
W := \frac{\beta}{\alpha} Y \quad \text{and} \quad Z := \frac{X}{\sqrt{W}}.
\end{align*}
Then we have $W \sim \Gamma(\alpha,\alpha)$, and hence $Z \sim t_{2 \alpha}$. Moreover, we obtain
\begin{align*}
\frac{X}{\sqrt{Y}} = \sqrt{\frac{\beta}{\alpha}} \frac{X}{\sqrt{W}} = \sqrt{\frac{\beta}{\alpha}} Z,
\end{align*}
completing the proof.
\end{proof}

We remark that the distribution of $X / \sqrt{Y}$ is a Pearson type VII distribution; cf. \cite[p. 450]{Pearson}. More precisely, the \emph{Pearson type VII distribution} ${\rm P}(\alpha,m)$ with parameters $\alpha \in (0,\infty)$, $m \in (\frac{1}{2},\infty)$ is the absolutely continuous probability measure on $(\bbr,\calb(\bbr))$ with density
\begin{align}
f(x) = \frac{\Gamma(m)}{\sqrt{\pi} \alpha \Gamma(m - \frac{1}{2})} \bigg( 1 + \frac{x^2}{\alpha^2} \bigg)^{-m}, \quad x \in \bbr.
\end{align}
Note that $t_{\nu} = {\rm P}(\sqrt{\nu}, (\nu+1) / 2)$ for each $\nu > 0$. Furthermore, for $Z \sim {\rm P}(\alpha,m)$ and $c > 0$ we have $cZ \sim {\rm P}(c\alpha,m)$. Therefore, in the situation of Proposition \ref{prop-Pearson} we obtain
\begin{align*}
\frac{X}{\sqrt{Y}} \sim {\rm P} \bigg( \sqrt{2 \beta}, \alpha + \frac{1}{2} \bigg).
\end{align*}

\begin{proposition}\label{prop-F-distribution}
For two independent random variables independent $X \sim \Gamma(\alpha,\beta)$ and $Y \sim \Gamma(\gamma,\delta)$ we have
\begin{align*}
\frac{X}{Y} =  \frac{\alpha \delta}{\beta \gamma} Z
\end{align*}
with a random variable $Z \sim {\rm F}_{2 \alpha, 2 \gamma}$.
\end{proposition}

\begin{proof}
We define the random variables
\begin{align*}
V := \frac{\beta}{\alpha} X, \quad W := \frac{\delta}{\gamma} Y \quad \text{and} \quad Z := \frac{V}{W}.
\end{align*}
Then we have $V \sim \Gamma(\alpha,\alpha)$ and $W \sim \Gamma(\gamma,\gamma)$, and hence $Z \sim {\rm F}_{2 \alpha, 2 \gamma}$. Moreover, we obtain
\begin{align*}
\frac{X}{Y} = \frac{\alpha \delta}{\beta \gamma} \frac{V}{W} = \frac{\alpha \delta}{\beta \gamma} Z,
\end{align*}
completing the proof.
\end{proof}

We remark that the distribution of $X/Y$ is a generalized Fisher distribution; cf. \cite[Sec. 8.1]{Johnson-Kotz}. More precisely, the \emph{generalized Fisher distribution} ${\rm F}_{m,n}(a,b)$ with parameters $m,n > 0$ and $a,b>0$ is the absolutely continuous probability measure on $(\bbr,\calb(\bbr))$ with density
\begin{align*}
f(x) = \frac{(\frac{m}{n})^\frac{m}{2} (ab)^{-1}}{B(\frac{m}{2},\frac{n}{2})} \frac{(\frac{x}{a})^{\frac{m}{2b}-1}}{\big(1 + \frac{m}{n} (\frac{x}{a})^{\frac{1}{b}}\big)^{\frac{m+n}{2}}} \bbI_{(0,\infty)}(x),
\end{align*}
where $B : (0,\infty) \times (0,\infty) \to (0,\infty)$ denotes the Beta function. Note that ${\rm F}_{m,n}(1,1) = {\rm F}_{m,n}$ for all $m,n > 0$. Furthermore, for $Z \sim {\rm F}_{m,n}(a,b)$ and $c > 0$ we have $cZ \sim {\rm F}_{m,n}(ca,b)$. Therefore, in the situation of Proposition \ref{prop-F-distribution} we obtain
\begin{align*}
\frac{X}{Y} \sim {\rm F}_{2 \alpha, 2 \gamma} \bigg( \frac{\alpha \delta}{\beta \gamma}, 1 \bigg).
\end{align*}

\end{appendix}

\end{document}